\documentclass[10pt,a4paper]{article}
\usepackage[active]{srcltx}

\usepackage[final]{optional}
\usepackage[british,american]{babel}

\usepackage{amsfonts, amsmath, wasysym}
\ifx\pdfoutput\undefined
\usepackage{graphicx}
\else
\usepackage[pdftex]{graphicx}
\usepackage{epstopdf}
\fi

\usepackage{amssymb,amsthm,
  paralist
}
\usepackage{
  latexsym,
  nicefrac,
}
\usepackage[usenames]{color}
\usepackage{url}

\definecolor{darkgreen}{rgb}{0,0.5,0}
\definecolor{darkred}{rgb}{0.7,0,0}
\usepackage[colorlinks,
citecolor=darkgreen, linkcolor=darkred
]{hyperref}
\usepackage{esint}
\usepackage{bibgerm}
\usepackage[normalem]{ulem}

\theoremstyle{plain}
\newtheorem{lemma}{Lemma}[section]
\newtheorem{thm}[lemma]{Theorem}
\newtheorem{prop}[lemma]{Proposition}
\newtheorem{cor}[lemma]{Corollary}

\theoremstyle{definition}
\newtheorem{defn}[lemma]{Definition}

\newtheorem{rmk}[lemma]{Remark}
\numberwithin{equation}{section}
\newcommand{\m}{\ensuremath{{\cal M}}}

\newcommand{\pl}[2]{{\frac{\partial #1}{\partial #2}}}

\newcommand{\be}{\beta}

\newcommand{\Om}{\Omega}

\newcommand{\ep}{\varepsilon}

\newcommand{\R}{\ensuremath{{\mathbb R}}}
\newcommand{\N}{\ensuremath{{\mathbb N}}}

\newcommand{\C}{\ensuremath{{\mathbb C}}}

\newcommand{\downto}{\downarrow}
\newcommand{\upto}{\nearrow}

\newcommand{\lap}{\Delta}

\DeclareMathOperator{\Vol}{vol}

\newcommand{\beq}{\begin{equation}}
  \newcommand{\eeq}{\end{equation}}
\newcommand{\beqa}{\begin{equation}\begin{aligned}}
    \newcommand{\eeqa}{\end{aligned}\end{equation}}
\newcommand{\brmk}{\begin{rmk}}
  \newcommand{\ermk}{\end{rmk}}
\newcommand{\partref}[1]{\hbox{(\csname @roman\endcsname{\ref{#1}})}}
\newcommand{\half}{\frac{1}{2}}

\newcommand{\Ric}{{\mathrm{Ric}}}

\newcommand*{\ee}{\mathop{\mathrm{e}}\nolimits}

\newcommand*\spt{\mathop{\mathrm{spt}}\nolimits}

\newcommand*\dist{\mathop{\mathrm{dist}}\nolimits}

\newcommand*\arsinh{\mathop{\mathrm{arsinh}}\nolimits}

\newcommand*\Mf{\mathcal{M}}
\newcommand*\Disc{\mathcal{D}}

\newcommand*\gBall{\mathcal{B}}
\newcommand*\Sph{\mathcal{S}}
\newcommand*\rSph{\mathbb{S}}

\newcommand*\pddt{\frac{\partial}{\partial t}}

\newcommand*\dx{\mathrm{d}x}

\newcommand*\dell{\mathrm{d}\ell}
\newcommand*\dy{\,\mathrm{d}y}

\newcommand*\dtheta{\mathrm{d}\theta}
\newcommand*\dttheta{\mathrm{d}\tilde\theta}

\newcommand*\dz{\mathrm{d}z}

\title{{\sc Ricci flows with unbounded curvature}
}
\author{Gregor Giesen and Peter M. Topping}
\date{\today}

\begin{document}
\maketitle

\begin{abstract}
  We show that any noncompact Riemann surface admits a complete Ricci flow $g(t)$, $t\in [0,\infty)$, which has unbounded curvature for all $t\in [0,\infty)$.
\end{abstract}

\section{Introduction}

A Ricci flow is a smooth family of Riemannian metrics $g(t)$ on a manifold $\Mf$, for $t$ in some time interval, which satisfies the nonlinear PDE
\begin{equation}
  \label{eq:ricci-flow}
  \pddt g(t) = -2\Ric\bigl[g(t)\bigr].
\end{equation}
According to Hamilton \cite{Ham82} and Shi \cite{Shi89}, given a complete bounded curvature metric $g_0$ on $\Mf$, there exists a Ricci flow $g(t)$ with $g(0)=g_0$, for $t\in [0,T]$ (some $T>0$) which is complete and of bounded curvature (at each time $t\in [0,T]$).
In this paper, we will consider the manifold $\Mf$ to be a
surface, and in this case the flow preserves the conformal structure
(see Section \ref{ssect:ev-cf}).

A rich theory has been developed for this flow, most of which collapses if we drop the assumption of bounded curvature within the flow.
However, Ricci flows with potentially unbounded curvature arise naturally.
Topping \cite{Top10} and Giesen-Topping \cite{GT11} developed a theory (see Theorem \ref{thm:ic-existence} below) which allows one to flow a completely general surface, and in particular one of unbounded curvature. Clearly such a flow would then have unbounded Gaussian curvature $K$ in the sense that
$$\sup_{\Mf\times(0,T]} |K[g(t)]|=\infty,$$
say, but possibly only because the curvature is blowing up as $t\downto 0$. A natural question is therefore:
\begin{quote}
  {\em
    Can there exist a complete Ricci flow $g(t)$, $t\in [0,T]$, on a surface, with unbounded curvature for some $t\in (0,T)$, i.e. $\sup_{\Mf} |K[g(t)]|=\infty$?
  }
\end{quote}

It turns out that 2D Ricci flow has a mechanism for ensuring that the curvature \emph{cannot} be unbounded below at later times. Indeed, in \cite{Che09} B.-L. Chen
proves a very general apriori estimate for the scalar curvature of a Ricci flow (in any dimension)
without requiring anything but the completeness of the
solution, and in 2D this implies:
\begin{thm}{\rm(B.-L. Chen \cite{Che09}.)}
  Let $\bigl(g(t)\bigr)_{t\in(0,T]}$ be a
  complete Ricci flow on a surface $\Mf^2$. Then
  \[ K\bigl[g(t)\bigr] \ge -\frac1{2t}\qquad\text{for all }t\in(0,T]. \]
\end{thm}

Note that under Ricci flow, the Gaussian curvature evolves
(see e.g. \cite[Proposition 2.5.4]{Top06}) according to
$$\pl{K}{t}=\lap K +2K^2,$$
so if we knew that the curvature were bounded, then the maximum principle would apply and we could deduce Chen's result easily. However, Ricci flows with unbounded curvature are much more subtle, and Chen requires deep ideas from the work of Perelman to prove this uniform lower curvature bound. It is interesting to dwell on how Ricci flow stretches regions of extreme negative curvature (reducing the curvature) and keeps uncontrolled negative curvature `out at infinity', in order to ensure that this lower bound holds.

While Chen's result gives a good lower bound for the curvature, the question remains as to whether the curvature can ever be unbounded above. The purpose of this paper is to demonstrate that the answer is emphatically yes:

\vskip 10pt

\begin{thm}{\rm \bf(Main theorem.)}
  \label{thm:ex-unbounded-curv}
  On every noncompact Riemann surface $\Mf^2$ there exists a complete immortal Ricci flow
  $\bigl(g(t)\bigr)_{t\in[0,\infty)}$ with unbounded curvature $\sup_\Mf
  K[g(t)]=\infty$ for all $t\in[0,\infty)$.
\end{thm}

\vskip 10pt

\begin{rmk}
  By taking the Cartesian product of the flows in
  Theorem \ref{thm:ex-unbounded-curv} with $\R^n$, one can construct examples 
  of complete Ricci flows with unbounded curvature in higher dimensions. 
  However, prior to our work Esther
  Cabezas-Rivas and Burkhard Wilking have shown us constructions in
  higher dimensions where one has Ricci flows with unbounded curvature for all
  times which also have positive curvature. See \cite{CW11}.
\end{rmk}
The essential idea behind the proof of Theorem \ref{thm:ex-unbounded-curv} is to construct an initial metric which has a sequence of patches within it that become more and more curved, and are inclined under Ricci flow to retain their high curvature for longer and longer.
We will require a combination of pseudolocality arguments, barrier arguments and an isoperimetric inequality in order to establish appropriate lower curvature bounds for very large times, but even \emph{starting} the Ricci flow with such an uncontrolled metric
has only been possible since recently as a result of the following theorem.

\begin{thm}{\rm (Fragment of \cite[Theorem 1.3]{GT11}.)}
  \label{thm:ic-existence}
  Let $\bigl(\Mf^2,g_0\bigr)$ be any smooth Riemannian surface --
  it need not be
  complete, and could have unbounded curvature.
  Depending on the conformal type, we define $T\in(0,\infty]$ by
  \[ T := \begin{cases}
    \frac1{8\pi}\Vol_{g_0}\Mf & \text{if }(\Mf,g_0)\cong\Sph^2,\\
    \frac1{4\pi}\Vol_{g_0}\Mf & \text{if }(\Mf,g_0)\cong\mathbb
    C\text{ or
    }(\Mf,g_0)\cong\mathbb R\!P^2, \\
    \qquad\infty & \text{otherwise}.
  \end{cases} \]
  Then there exists a smooth Ricci flow $\bigl(g(t)\bigr)_{t\in[0,T)}$
  such that
  \begin{compactenum}
  \item $g(0)=g_0$;
  \item $g(t)$ is instantaneously complete (i.e. complete for all $t\in (0,T)$);
  \item $g(t)$ is maximally stretched\footnote{We call a Ricci flow
      $\bigl(g(t)\bigr)_{t\in[0,T)}$ \textbf{maximally stretched}
      provided that if $\bigl(\tilde g(t)\bigr)_{t\in[0,\tilde T)}$
      is another Ricci flow with $\tilde g(0)\le g(0)$ then $\tilde
      g(t)\le g(t)$ for all $t\in[0,\min\{T,\tilde T\})$.},
  \end{compactenum}
  and this flow is unique in the sense that if
  $\bigl(\tilde g(t)\bigr)_{t\in[0,\tilde T)}$ is
  any other Ricci flow on $\m$ satisfying 1,2 and 3, then
  $\tilde T\leq T$ and $\tilde g(t)=g(t)$ for all $t\in[0,\tilde T)$.
  If $T<\infty$, then we have
  \[ \Vol_{g(t)}\Mf = \left\{\begin{array}{ll}
      8\pi (T-t) & \text{if }(\Mf,g_0)\cong\Sph^2,\\
      4\pi (T-t) & \text{otherwise},
    \end{array}\right\}\quad\longrightarrow\quad 0 \quad\text{ as } t\nearrow
  T, \]
  and in particular, $T$ is the maximal existence time.
\end{thm}

\begin{rmk}
  The full theorem from \cite{GT11} includes a description of the
  asymptotics of the flow in partial generality.
  The case that $\Mf$ is compact
  was already known to Hamilton \cite{Ham88} and Chow
  \cite{Cho91}; in particular, as we will need later, 
  they show that under Ricci flow any
  sphere $\bigl(\Sph^2,g_0\bigr)$ converges after rescaling
  to the round sphere.
\end{rmk}

The article is organised as follows: We start by summarising some
specific properties of the two-dimensional Ricci flow
(\S~\ref{ssect:ev-cf}) and known facts about the cigar and its
associated Ricci flow solution (\S~\ref{ssec:cigar}).
In the following main part (\S~\ref{ssec:rf-cont-cigars}) we
prove some facts about Ricci flows starting at a surface
containing a truncated cigar. The main result here will be that the cigar is
mostly preserved for as long as we want regardless of the geometry of
the surrounding surface --- provided the cigar was initially long
enough. The proof involves the construction of suitable barriers using
the maximum principle in conjunction with a version of Perelman's
Pseudolocality theorem due to B.-L. Chen (Theorem \ref{thm:chen-pseudo}).
Exploiting these barriers we will apply an isoperimetric inequality by
G. Bol (Theorem \ref{thm:bol-isop-ineq}) to show that the
maximum of the curvature in that cigar region is universally bounded
from below by a suitable positive constant for as long as we want. Finally, in Section
\ref{sec:ricci-flow-unbounded} we conclude the proof of the Main
Theorem \ref{thm:ex-unbounded-curv} by patching a sequence of longer
and longer shrunken cigars onto an arbitrary Riemannian surface such that both
its curvature is unbounded and if we flow it using Theorem
\ref{thm:ic-existence} the results of Section \ref{sect:rf-cigar}
ensure that the curvature stays unbounded.
We remark that the earlier construction of Cabezas-Rivas and Wilking
\cite{CW11} in higher dimensions is also constructed out of cigars.
\medskip

{\em Acknowledgements:} Both authors were supported by The Leverhulme
Trust. Thanks to Esther Cabezas-Rivas for discussions about \cite{CW11}. 

\section{Ricci flows of surfaces and cigars}
\label{sect:rf-cigar}

\subsection{Evolution of the conformal factor}
\label{ssect:ev-cf}

On a two-dimensional manifold, the Ricci curvature is
simply the Gaussian curvature $K$ times the metric:
$\Ric[g]=K[g]\, g$. The Ricci flow then moves within a fixed conformal
class, and if we pick a local complex coordinate $z=x+\mathrm{i}y$
and write the metric in terms of a scalar conformal factor $u\in C^\infty(\Mf)$
\[ g=\ee^{2u}|\dz|^2 \]
where $|\dz|^2=\dx^2+\dy^2$, then the evolution of the metric's conformal
factor $u$ under Ricci flow is governed by the nonlinear scalar PDE
\begin{equation}
  \label{eq:ricci-flow-cf}
  \pddt u = \ee^{-2u}\Delta u = -K[u].
\end{equation}
where $\Delta := \frac{\partial^2}{\partial x^2} + \frac{\partial^2}{\partial
  y^2}$ is defined in terms of the local coordinates. If we talk about Ricci flow on a Riemann surface, then we require the conformal structures of the flow and the surface to coincide.

\subsection{The cigar}
\label{ssec:cigar}

In this section we collect some known results about Hamilton's cigar soliton.
For further details we refer to \cite[\S2.2.1]{CK04} or
\cite[\S4.3]{CLN06}.

On $\mathbb C$ one can write the cigar metric $g_\Sigma$ in terms of a global
complex coordinate $z=x+\mathrm iy$
\begin{equation}
  \label{eq:eucl-cigar}
  g_\Sigma(z) = \frac{|\dz|^2}{1+|z|^2}\quad\text{with
    Gaussian curvature}\quad K[g_\Sigma](z) =
  \frac{2}{1+|z|^2}.
\end{equation}
The centric geodesic ball of radius $r>0$ corresponds in
complex coordinates to
\[ \gBall_{g_\Sigma} (0;r) = \Disc_{\sinh r} := \bigl\{ z\in\mathbb C:
|z| < \sinh r\bigr\}. \]
Being a steady Ricci soliton, one can associate a self-similar
Ricci flow -- the \emph{cigar solution}
$\bigl(g_\Sigma(t)\big)_{t\in\mathbb R}$ on $\mathbb C$ --
defined by
\begin{equation}
  \label{eq:cigar-rf}
  g_\Sigma(t;z) = \ee^{2u_\Sigma(t,z)}|\dz|^2 =
  \frac{|\dz|^2}{\ee^{4t}+|z|^2}
  \quad\Longrightarrow\quad u_\Sigma(t,z) =
  -\frac12\log\left(\ee^{4t}+|z|^2\right).
\end{equation}
After puncturing at its tip $z=0$, we write the cigar in
cylindrical coordinates $(\ell,\theta)\in\mathbb R\times\rSph^1 =:
\mathcal{C}$ defined by $z= \ee^{\ell+\mathrm i\theta}$ as
\begin{equation}
  \label{eq:cyl-cigar}
  g_\Sigma(\ell,\theta) = \frac{\dell^2+\dtheta^2}{\ee^{-2\ell}+1}
  \le \dell^2+\dtheta^2.
\end{equation}
For large $|z|=\ee^\ell$ the cigar is almost a flat cylinder and therefore we expect the geodesic distance from its tip at
$\ell=-\infty$ to a point $(\ell,\theta)$ to be roughly $\ell$ for large $\ell$.
In fact we have
\begin{align}
  \label{eq:cigar-dist-cyl}
  \ell+\log2 \le \dist_{g_\Sigma}\bigl( (-\infty,\theta),(\ell,\theta)
  \bigr)  &= \arsinh\ee^\ell \le \ell+1 &&\text{for all }\ell\ge0.
\end{align}
Similarly, the area of a ball of large radius $r>0$ around the tip
should be roughly the area of a cylinder of length $\ell\sim r$,
i.e. $\Vol_{g_\Sigma} \sim 2\pi\ell$; more precisely we have the lower
estimates
\begin{align}
  \label{eq:cigar-area-cyl}
  \Vol_{g_\Sigma}\bigl( (-\infty,\ell)\times\rSph^1 \bigr) &=
  2\pi\log\cosh\arsinh\ee^\ell \ge 2\pi\ell &&\text{for all
  }\ell\ge0\\
  \label{eq:cigar-area-cyl-r}
  \text{and}\qquad\Vol_{g_\Sigma}\bigl(\gBall_{g_\Sigma}(0;r)\bigr) &=
  2\pi \log\cosh r \ge 2\pi(r-\log2) &&\text{for all }r\ge0.
\end{align}
The following two estimates quantify further the assertion above
that the cigar is asymptotically a flat cylinder: For any $r>0$ and for
any point in $\mathbb C\setminus\gBall_{g_\Sigma}(0;r)$, i.e. for
all
$(\ell,\theta)\in\bigl(\log\sinh r,\infty\bigr)
\times\rSph^1
\subset\mathcal C$, we have the rough estimates for the metric
\begin{equation}
  \label{eq:cigar-metric-flat}
  \left(1-\frac1{r^2}\right)\left(\dell^2+\dtheta^2\right) \le
  \bigl(\tanh r\bigr)^2\left(\dell^2+\dtheta^2\right)
  \le g_\Sigma(\ell,\theta) \le \dell^2+\dtheta^2
\end{equation}
(because $\bigl(\tanh r\bigr)^2\left(\dell^2+\dtheta^2\right)=g_\Sigma(\log\sinh r,\theta)$)
and for the Gaussian curvature
\begin{equation}
  \label{eq:cigar-curv-flat}
  \sup_{(\log\sinh r,\infty)
    \times\rSph^1} K[g_\Sigma] =
  \frac2{(\cosh r)^2}\le \frac2{r^2}.
\end{equation}
In cylindrical coordinates the self-similarity of the cigar
solution is more obvious: under the Ricci flow it just translates:
\begin{equation}
  \label{eq:cigar-rf-cyl}
  g_\Sigma(t;\ell,\theta) =
  \frac{\dell^2+\dtheta^2}{\ee^{4t-2\ell}+1} = g_\Sigma(\ell-2t,\theta).
\end{equation}
Consequently, thinking again of it as a cylinder for sufficiently
large $\ell>2t$ which now translates in time, we expect the distance to
its tip  to behave like $\sim \ell-2t$, and similarly the area of the
$g_\Sigma$-geodesic ball, centred at the tip, of radius $r$ should behave like
the area of a cylinder of length $r-2t$, i.e. $\sim2\pi(r-2t)$;
more precisely we have the lower estimates
\begin{align}
  \label{eq:cigar-dist-t}
  \dist_{g_\Sigma(t)}\bigl(0, \partial\gBall_{g_\Sigma}(0;r)\big)
  &= \arsinh\ee^{(\log\sinh r) - 2t} \ge r-2t\qquad\qquad\text{and}\\
  \label{eq:cigar-area-t}
  \Vol_{g_\Sigma(t)}\gBall_{g_\Sigma}(0;r) &= 2\pi \log\cosh\arsinh
  \ee^{(\log\sinh r)-2t} \ge 2\pi\bigl( r-2t-\log2 \big).
\end{align}

\subsection{Ricci flows containing cigars}
\label{ssec:rf-cont-cigars}

The Ricci flow we construct to prove Theorem
\ref{thm:ex-unbounded-curv} will start with a metric containing countably many truncated cigars, at smaller and smaller scales, and which are longer and longer.
The proof will rely on each truncated cigar evolving much like a whole cigar would evolve, for a very long time, irrespective of how wild the metric is beyond the cigar part. However, great care is required to establish such behaviour, as is indicated by the following example.
\begin{thm} 
  \label{thm:sucked_out1}
  Given $R>0$ (however large) and $\ep>0$ (however small) there exists a Ricci flow $g(t)$ on $\Disc_{\sinh R}\subset\C$ for $t\in [0,\ep)$ such that $g(0)=g_\Sigma$ (a truncated cigar of arbitrarily long length $R$) but for which
  $$\Vol_{g(t)}\bigl(\Disc_{\sinh R}\bigr)\longrightarrow 0 \quad\text{ as }t\upto \ep.$$
\end{thm}
In other words, Ricci flow can suck out an arbitrarily large amount of area from a remote part of the manifold, in an arbitrarily short period of time. The Ricci flow of Theorem \ref {thm:sucked_out1} can be constructed based on  Theorem \ref{thm:sucked_out2} in the appendix via an appropriate barrier argument.

Despite this extreme behaviour, the results of this section show that a truncated cigar evolving as part of a larger \emph{complete} Ricci flow cannot suffer such a fate, and must evolve like a complete cigar for an arbitrarily long time.
\begin{defn}
  \label{defn:contained-cigar}
  A surface $\bigl(\Mf^2,g\bigr)$ contains at the point
  $p\in\Mf$ a cigar of length $R>0$ at scale $\alpha>0$
  if $\gBall_g(p;R)\Subset\Mf$ and there exists an isometric
  diffeomorphism $\psi$ from $\gBall_g(p;R)$ to the cigar at scale
  $\alpha$ (i.e. the rescaled cigar $\alpha^2g_\Sigma$)
  restricted to the disc
  $\Disc_{\sinh(\alpha^{-1}R)}\subset\mathbb C$:
  \[ \psi: \Bigl(\gBall_g(p;R),g\Bigr) \longrightarrow
  \Bigl(\Disc_{\sinh(\alpha^{-1}R)}, \alpha^2g_\Sigma\Bigr).\]
\end{defn}
The next lemma obtains curvature control at a point part of the way down an evolving truncated cigar, at least while the point remains well within the interior of the evolving manifold and the tip does not get too close.

\begin{lemma}
  \label{lemma:bdd-curv-at-R}
  For some universal constant $B>0$ and for all times $T>0$, radii
  $r_0\ge\sqrt{BT}$ and lengths $\tilde R\ge 4(r_0+1)$, if
  $\bigl(g(t)\bigr)_{t\in[0,T]}$ is Ricci flow on a noncompact Riemann
  surface $\Mf^2$ such that $\bigl(\Mf,g(0)\bigr)$ contains at the
  point $p\in\Mf$ a cigar of length $2\tilde R$ at unit scale, then we have
  \begin{equation}
    \label{eq:bdd-curv-at-R}
    \sup_{\partial\gBall_{g(0)}(p;\tilde R)} K[g(t)] \le 2r_0^{-2}
    \qquad\text{for all }t\in[0,\tau]
  \end{equation}
  where
  \begin{equation}
    \label{eq:defn-tau}
    \tau := \sup\left\{ t_0\in[0,T] :
      \begin{array}{l}
        \gBall_{g(t)}(q;r_0) \Subset
        \Mf\setminus\{p\} \\
        \text{for all } q\in\partial\gBall_{g(0)}(p;\tilde R) \text{ and }
        t\in[0,t_0]
      \end{array}\right\}>0.
  \end{equation}
\end{lemma}
The lemma exploits Chen's two-dimensional version of Perelman's Pseudolocality
theorem (Theorem \ref{thm:chen-pseudo}) to gain arbitrary
control of the curvature ($r_0$ can be chosen as large as we want).
Since we
are about to apply that theorem on a region where
the cigar is almost a cylinder, the non-collapsed condition (iii)
would cause trouble for large $r_0$. Therefore we will puncture the
cigar at its tip and lift it to its universal cover (being locally a
half plane) such that the problem vanishes and we can establish the
curvature estimate for the lifted Ricci flow. The restriction to
$[0,\tau]$ in \eqref{eq:bdd-curv-at-R}
is basically the price to pay for that trick to still fulfil condition (i)
of that theorem on the universal cover of the punctured surface
$\Mf\setminus\{p\}$. However this turns out not to be a restriction
because the barriers we construct in the next Lemma
\ref{lemma:cigar-barriers} will ensure that for
sufficiently large $\tilde R$ we will have maximal $\tau=T$ if we additionally
require the Ricci flow $g(t)$ to be instantaneously complete.

\begin{proof}
  For $v_0=\frac{3\pi}4$ we obtain a universal $B>0$ as the constant
  $C=C(v_0)>0$ in Theorem \ref{thm:chen-pseudo}. Fix $T>0$,
  $r_0\ge\sqrt{BT}$ and $\tilde R\ge 4(r_0+1)$. Let
  $\bigl(g(t)\bigr)_{t\in[0,T]}$ be a Ricci flow such that
  $\bigl(\Mf,g(0)\bigr)$ contains at the point $p\in\Mf$
  a cigar of length $2\tilde R$ at unit scale.

  By virtue of the uniformisation theorem, since $\Mf$ is noncompact its
  universal cover is conformally the plane $\mathbb C$ or the disc
  $\Disc$. Because it suffices to establish the
  curvature estimate \eqref{eq:bdd-curv-at-R} on any sheet of its universal
  cover, we can assume without loss of generality that $\Mf$ is
  conformal to $\mathbb C$ or $\Disc$ and
  in particular, punctured at the tip of the cigar,
  $\Mf_p:=\Mf\setminus\{p\}$ is homeomorphic to a cylinder.

  Using (possibly a second time) the universal covering
  $\pi:\widetilde{\Mf_p}\to\Mf_p$ to lift the Ricci flow $\tilde
  g(t):=\pi^*g(t)$ for all $t\in[0,T]$, we can write the lifted
  punctured cigar $\tilde g(0)$ on
  $\pi^{-1}\bigl(\gBall_{g(0)}(p;2\tilde R)\bigr)$ locally in `lifted'
  cylindrical coordinates $(\ell,\tilde\theta)\in(-\infty,\log\sinh
  2\tilde R)\times\mathbb R$  as
  \begin{equation*}
    \label{eq:lifted-cyl-coordinates}
    \tilde g(0;\ell,\theta) = \frac{\dell^2+\dttheta^2}{\ee^{-2\ell}+1}.
  \end{equation*}
  Now fix any point
  $q\in\pi^{-1}\left(\partial\gBall_{g(0)}\bigl(p;\tilde R\bigr)\right)$.
  We want to apply Theorem \ref{thm:chen-pseudo} to $\bigl(\tilde
  g(t)\big)_{t\in[0,\tau)}$ on $\gBall_{\tilde g(t)}(q;r_0)$.
  By definition of $\tau$, we have
  $\gBall_{\tilde g(t)}(q;r_0) \Subset
  \widetilde{\Mf_p}$ for all $t\in[0,\tau)$,
  and therefore condition (i) of that theorem is fulfilled. As
  $\tilde R\ge 4(r_0+1)$ we certainly have $\gBall_{\tilde
    g(0)}(q;r_0)\subset\pi^{-1}\left(
    \gBall_{g(0)}(p;2\tilde R)\setminus\gBall_{g(0)}(p;\nicefrac{\tilde R}2)
  \right)\simeq\bigl(\log\sinh\nicefrac{\tilde R}2,\log\sinh 2\tilde R)\times\mathbb R$  by
  equation \eqref{eq:cigar-dist-cyl}, and
  condition (ii) follows using estimate \eqref{eq:cigar-curv-flat}
  \[ \sup_{\gBall_{\tilde g(0)}(q;r_0)} \Bigl|K[\tilde g(0)]\Bigr|
  \le \sup_{(\log\sinh\nicefrac{\tilde R}2,\infty) \times\rSph^1} K[g_\Sigma]
  \le 8 \tilde R^{-2} \le r_0^{-2}. \]
  Similarly, transferring \eqref{eq:cigar-metric-flat} to this case i.e.
  \[  \frac34 \bigl(\dell^2+\dttheta^2\bigr)
  \le\tilde g(0)\le \dell^2+\dttheta^2 \qquad\text{on }
  \gBall_{\tilde g(0)}(q;r_0) \]
  we have $\gBall_{\dell^2+\dttheta^2}(q;r_0)\subset\gBall_{\tilde
    g(0)}(q;r_0)$ and condition (iii):
  \[ \Vol_{\tilde g(0)}\gBall_{\tilde g(0)}(q;r_0) \ge
  \frac34 \Vol_{\dell^2+\dttheta^2} \gBall_{\dell^2+\dttheta^2}(q;r_0)
  = \frac{3\pi}4 r_0^2. \]
  Applying Theorem \ref{thm:chen-pseudo} for every
  $q\in\pi^{-1}\left(\partial\gBall_{g(0)}(p;\tilde R)\right)$ yields the
  desired curvature estimate for $\bigl(\tilde g(t)\bigr)_{t\in[0,\tau)}$ on
  $\pi^{-1}\left(\partial\gBall_{g(0)}(p;\tilde R)\right)$, and therefore
  also for $\bigl(g(t)\bigr)_{t\in[0,\tau]}$ on
  $\partial\gBall_{g(0)}(p;\tilde R)$ by continuity.
\end{proof}

The curvature estimate of Lemma \ref{lemma:bdd-curv-at-R} can be
transformed into a constraint on the conformal factor of $g(t)$
on $\partial \gBall_{g(0)}(p;\tilde R)$
which will allow us to apply a maximum principle to
establish local barriers (see \eqref{eq:cigar-barriers}) for
$t\in [0,\tau]$. Once we have the
lower barrier we can utilise it to show that the Ricci flow does not
shrink the tip of the cigar too fast; hence $\tau$ in Lemma
\ref{lemma:bdd-curv-at-R} will have to be maximal ($\tau=T$) and we
will have upper and lower barriers for all $t\in [0,T]$.

\begin{lemma}
  \label{lemma:cigar-barriers}
  For some universal constants $\beta>1,A>0$ and
  for all times $T>0$ and lengths $\tilde R\ge A(T+1)$,
  if $\bigl(g(t)\bigr)_{t\in[0,T]}$ is
  an instantaneously complete Ricci flow on a noncompact Riemann
  surface $\Mf^2$ such that $\bigl(\Mf,g(0)\bigr)$ contains at the
  point $p\in\Mf$ a cigar of length $2\tilde R$ at unit scale, then
  there exist cigar solutions $\bigl(g_\pm(t)\bigr)_{t\in\mathbb R}$
  at scales $\beta^{\pm1}$
  which are locally upper and lower barriers for
  $g(t)$ in the sense that on $\gBall_{g(0)}(p;\tilde R)$
  \begin{equation}
    \label{eq:cigar-barriers}
    g_-(t):= \psi^*\Bigl(\beta^{-2}g_\Sigma\bigl(\beta^2t\bigr)\Bigr) \le
    g(t) \le \psi^*\Bigl(\beta^2
    g_\Sigma\bigl(\beta^{-2}(t-T)\bigr)\Bigr)=:g_+(t)
  \end{equation}
  for all $t\in[0,T]$ with $\psi$ as in
  Definition \ref{defn:contained-cigar}.
\end{lemma}
\begin{proof}
  Let $B$ be the universal constant from Lemma
  \ref{lemma:bdd-curv-at-R}, fix $T>0$, and define $r_0=\sqrt{BT}$ and
  $\beta=\ee^{\frac2B}>1$. Then we can find a universal $A>0$ such
  that $A(T+1)\ge\max\bigl\{ \beta(2\beta T+r_0+1),
  4(r_0+1)\bigr\}$ and choose an arbitrary $\tilde R\ge A(T+1)$. Note
  that this way $\tilde R$ fulfils the requirement of Lemma
  \ref{lemma:bdd-curv-at-R}.
  Finally define $\tau\in[0,T]$ according to \eqref{eq:defn-tau}.

  Without loss of generality we can write $\psi_*
  g(t)\bigr|_{\gBall_{g(0)}(p;\tilde R)}=\ee^{2u(t)}|\dz|^2$ on
  $\overline{\Disc_{\sinh \tilde R}}$. Using \eqref{eq:cigar-rf}
  we obtain the conformal factors of the proposed barriers
  $\psi_*\bigl(g_\pm(t)\bigr)$ of
  \eqref{eq:cigar-barriers}
  \[   u_-(t,z) = u_\Sigma\bigl(\beta^2t,z\bigr)-\log\beta
  \qquad\text{and}\qquad
  u_+(t,z) = u_\Sigma\bigl(\beta^{-2}(t-T),z\bigr)+\log\beta. \]
  To establish \eqref{eq:cigar-barriers} we are going to exploit the
  standard parabolic maximum principle for solutions of
  \eqref{eq:ricci-flow-cf} on $[0,\tau]\times\overline{\Disc_{\sinh \tilde R}}$:
  Since by assumption $u(0,z) = u_\Sigma(0,z)$ for all $z\in\Disc_{\sinh \tilde R}$,
  we have initially
  \[ u_-(0,z) = u_\Sigma(0,z) -\log\beta \le  u(0,z) \le
  u_\Sigma(-\beta^{-2}T,z) + \log\beta = u_+(0,z) \]
  for all $z\in\overline{\Disc_{\sinh \tilde R}}$.
  For the requirement on the boundary $[0,\tau]\times\partial\Disc_{\sinh
    \tilde R}$, observe that
  estimating (with Lemma \ref{lemma:bdd-curv-at-R}) and integrating
  the evolution equation
  \eqref{eq:ricci-flow-cf} for $u(t)$ yields for all $t\in[0,\tau]$
  \[ \left| \pddt u(t)\right| = \Bigl|K[g(t)]\Bigr| \le \frac2{BT} \le
  \frac2{B\tau}
  \quad
  \text{ and thus }
  \quad \Bigl|u(t) - u(0)\Bigr|  \le \frac2B = \log\beta
  \quad\text{on }\partial\Disc_{\sinh \tilde R}. \]
  Hence, we have for all $t\in[0,\tau]$ and $z\in\partial\Disc_{\sinh \tilde R}$
  \begin{align*}
    u(t,z) &\le u(0,z) + \log\beta
    \le u_\Sigma\bigl(\beta^{-2}(t-T),z\bigr) + \log\beta
    = u_+(t,z) \\
    \text{and}\qquad u(t,z) &\ge u(0,z) - \log\beta \ge
    u_\Sigma\bigl(\beta^2t,z\bigr) - \log\beta
    = u_-(t,z).
  \end{align*}
  That is sufficient to apply the maximum principle and conclude
  $$u_-(t,z)\le u(t,z)\le u_+(t,z),$$
  on $[0,\tau]\times \Disc_{\sinh \tilde R}$, which proves
  \eqref{eq:cigar-barriers} for all $t\in[0,\tau]$ after pulling back
  the metrics via $\psi$.

  Now we will use the lower barrier to show that in fact, with our
  choice of $\tilde R\ge A(T+1)$ we have maximal $\tau=T$:
  Since $\bigl(g(t)\bigr)_{t\in[0,T]}$ is instantaneously complete we
  may rephrase the definition of $\tau$ in \eqref{eq:defn-tau} to
  \begin{equation}
    \label{eq:defn-tau-ic}\tag{\ref{eq:defn-tau}'}
    \tau = \max\Bigl\{ t_0\in[0,T] :
    \dist_{g(t)}\Bigl(p,\partial\gBall_{g(0)}\bigl(p;\tilde R\bigr)\Bigr)\ge r_0
    \text{ for all } t\in[0,t_0]
    \Bigr\}.
  \end{equation}
  Now assume that $\tau<T$.
  With our choice of $\tilde R\ge\beta(2\beta T +r_0+1)$ we have
  \begin{align*}
    \dist_{g(t)}\bigl(p,\partial\gBall_{g(0)}(p;\tilde R)\bigr)
    &\ge \dist_{g_-(t)}\bigl(p,\partial\gBall_{g(0)}(p;\tilde R)\bigr)\\
    &= \dist_{\beta^{-2}g_\Sigma(\beta^2t)}\bigl( 0,
    \partial\gBall_{g_\Sigma}(0;\tilde R)\bigr)\\
    &\ge \beta^{-1}\left(\tilde R-2\beta^2 t\right) & \text{using
      \eqref{eq:cigar-dist-t}}    \\
    &\ge \beta^{-1}\tilde R - 2\beta  T \ge r_0 + 1
  \end{align*}
  for all $t\in[0,\tau]$, which contradicts the maximality
  of $\tau$ in \eqref{eq:defn-tau-ic} since $t\mapsto g(t)$ is
  continuous. Hence $\tau = T$.
\end{proof}

\begin{prop}
  \label{prop:lower-bd-curv-1}
  For some universal $\varepsilon>0$ and $A>0$, and for all times $T>0$ and
  lengths $R\ge A(T+1)$,
  if $\bigl(g(t)\bigr)_{t\in[0,T]}$ is an instantaneously
  complete Ricci flow on a noncompact Riemann surface $\Mf^2$ such that
  $\bigl(\Mf,g(0)\bigr)$ contains a cigar of length $R$ at unit scale,
  then
  \begin{equation}
    \label{eq:lower-bd-curv-1}
    \sup_\Mf K[g(t)] \ge \varepsilon\qquad\text{for all } t\in[0,T].
  \end{equation}
\end{prop}
\begin{proof}
  Let $\beta>1$ be the universal scaling for the barriers $g_\pm(t)$
  from Lemma \ref{lemma:cigar-barriers},
  and let $\tilde A>0$ be the $A$ from that same lemma.
  Fix a larger $A\ge2\tilde A$ such that $A(T+1)\ge 2\beta^2(2T+\beta^2+1)$ for
  all $T>0$. Now fix $T>0$ and pick
  $R\ge A(T+1)$. For improved legibility we are going to abbreviate $B_r :=
  \gBall_{g(0)}(p;r)$ for any $r>0$.
  To establish \eqref{eq:lower-bd-curv-1} we want to apply an
  isoperimetric inequality due to Bol (Theorem
  \ref{thm:bol-isop-ineq}) to the domain $B_{\rho(t)}$ of area
  $2\pi\beta^2$ around the tip of the cigar where
  \[ \rho(t) := \max\Bigl\{ r>0 : \Vol_{g(t)} B_r \le
  2\pi\beta^2\Bigr\}. \]
  To see that $\rho(t)\le\frac R2$ for all $t\in[0,T]$, i.e. the domain
  $B_{\rho(t)}$ stays in the region $B_{\frac R2}$ where
  Lemma \ref{lemma:cigar-barriers} provides
  the barriers $g_\pm(t)$, we can use the lower barrier $g_-(t)$ to
  estimate the area with respect to $g(t)$ from below:
  Keeping in mind $\frac R{2\beta^2}\ge 2T+\beta^2+1$ we estimate
  \begin{align*}
    \Vol_{g(t)} B_{\frac R2} &\ge \Vol_{g_-(t)} B_{\frac R2}
    = \Vol_{\beta^{-2}g_\Sigma(\be^2t)} \gBall_{g_\Sigma}(0;\nicefrac R2)\\
    &\ge 2\pi \beta^{-2}\left(\nicefrac R2-2\beta^2T-\log2 \right)
    & \text{using \eqref{eq:cigar-area-t}}\\
    &\ge 2\pi\beta^2
  \end{align*}
  for all $t\in[0,T]$.
  Meanwhile, the upper barrier allows us to estimate the length of the
  boundary $\partial B_{\rho(t)}$ independently of $t\in[0,T]$ by the
  circumference of the scaled cylinder
  \[ \mathrm L_{g(t)}\,\partial B_{\rho(t)}
  \le \mathrm L_{g_+(t)}\,\partial B_{\rho(t)}
  \le \mathrm L_{\beta^2g_\Sigma} \partial\gBall_{g_\Sigma}\bigl(0;\rho(t)\bigr)
  \le 2\pi\beta \]
  using \eqref{eq:cigar-metric-flat}.
  Now we can apply Bol's isoperimetric inequality, Theorem
  \ref{thm:bol-isop-ineq}, to conclude
  \begin{align*}
    \sup_\Mf K[g(t)] &\ge \sup_{B_{\rho(t)}} K[g(t)] \\
    &\ge \frac{4\pi}{\Vol_{g(t)} B_{\rho(t)}}
    - \frac{\left(\mathrm L_{g(t)}\,\partial B_{\rho(t)}\right)^2}{
      \left(\Vol_{g(t)}B_{\rho(t)}\right)^2} \\
    &\ge \frac{4\pi}{2\pi\beta^2} -
    \frac{(2\pi\beta)^2}{(2\pi\beta^2)^2}
    = \beta^{-2} =: \varepsilon.
  \end{align*}
\end{proof}

\begin{cor} 
  \label{cor:lower-bd-curv-a}
  For some universal $\varepsilon>0$ and $A>0$, and for all times $T>0$, scales
  $\alpha\in(0,1)$ and lengths $R\ge A\alpha^{-1}(T+1)$, if
  $\bigl(g(t)\bigr)_{t\in[0,T]}$ is an instantaneously
  complete Ricci flow on a noncompact Riemann surface $\Mf^2$ such that
  $\bigl(\Mf,g(0)\big)$ contains a cigar of length $R$ at scale
  $\alpha$, then
  \begin{equation}
    \label{eq:lower-bd-curv-a}
    \sup_\Mf K[g(t)] \ge \varepsilon\alpha^{-2} \qquad\text{for all }
    t\in[0,T].
  \end{equation}
\end{cor}
\begin{proof}
  Fix $T>0$ and $\alpha\in (0,1)$. From the preceding Proposition
  \ref{prop:lower-bd-curv-1} we obtain the universal constants
  $\varepsilon>0$ and $A>0$. Define $\bar T:=\alpha^{-2}T$ and
  observe that with $R\ge A\alpha^{-1}(T+1)\ge
  A\alpha(\alpha^{-2}T+1) = \alpha A(\bar T+1)$ we have $\bar
  R:=\alpha^{-1}R\ge A(\bar T+1)$ as required in Proposition
  \ref{prop:lower-bd-curv-1}.
  Let $g(t)$ be an
  instantaneously complete Ricci
  flow such that $\bigl(\Mf,g(0)\bigr)$ contains a cigar of length $R$
  at scale $\alpha$. Then $\bigl(\bar g(s)\bigr)_{s\in[0,\alpha^{-2}T]}$ with
  $\bar g(s)= \alpha^{-2} g(\alpha^2s)$ is
  another instantaneously complete Ricci flow such that
  $\bigl(\Mf,\bar g(0)\bigr)$ contains a cigar of length $\bar
  R=\alpha^{-1}R$ at unit scale and we may apply Proposition
  \ref{prop:lower-bd-curv-1} to conclude
  \[ \sup_\Mf K[g(t)] = \sup_\Mf K[\alpha^2\bar g(\alpha^{-2}t)] = \sup_\Mf\alpha^{-2} K[\bar
  g(\alpha^{-2}t)] \ge \alpha^{-2}\varepsilon \]
  for all $t\in[0,T]$.
\end{proof}

\section{Ricci flow of unbounded curvature}
\label{sec:ricci-flow-unbounded}
In this section we give a proof of Theorem \ref{thm:ex-unbounded-curv}.
It is clearly sufficient to find an instantaneously complete Ricci flow rather than a complete one since we can always adjust the flow a little in time (i.e. consider $g(t+\ep)$).

The strategy is to construct an appropriate $g(0)$ containing lots of cigars at different scales and of different lengths, then to flow the metric using Theorem \ref{thm:ic-existence} and apply the results of the previous section, and in particular Corollary \ref{cor:lower-bd-curv-a}, to show that the curvature is unbounded.

\begin{proof}[Proof of \textbf{Theorem \ref{thm:ex-unbounded-curv}}]
  Let $\bar g_0$ be any conformal metric on $\Mf$.
  Since $\Mf$ is a noncompact Riemann surface, there
  exists a sequence of pairwise
  disjoint, simply connected, open subsets
  $\bigl(U_k\bigr)_{k\in\mathbb N}\subset \Mf$ moving off to infinity in the sense that for all $\Om\Subset\Mf$, the sets $U_k$ are disjoint from $\Om$ for sufficiently large $k$. For the sequences
  of scales $\alpha_k := \frac1k$ and times $T_k:=k$, Corollary
  \ref{cor:lower-bd-curv-a} permits us to choose a monotonically
  increasing sequence of lengths $R_k=Ak(k+1)\sim k^2$ compatible with $\alpha_k$ and $T_k$. By virtue of the
  Uniformisation
  theorem, for each
  $k\in\mathbb N$ there exists a conformal diffeomorphism
  \[ \psi_k: U_k \longrightarrow
  \Disc_{2\sinh(\alpha_k^{-1}R_k)}\subset\mathbb C \]
  onto the complex disc of radius
  $2\sinh\bigl(\alpha_k^{-1}R_k\bigr)$. Abbreviate
  \[ V_k := \psi_k^{-1}\left(\Disc_{\sinh(\alpha_k^{-1}R_k)}\right)\Subset U_k \]
  and choose smooth cut-off functions $f_k\in
  C^\infty\bigl(\Mf,[0,1]\bigr)$ such that $\spt f_k \subset U_k$
  and $f_k\bigr|_{V_k} \equiv1$.
  We can then define a new smooth metric
  $g_0$ on $\Mf$ by
  \[ g_0 := \left(1-\sum_{k\in\mathbb N} f_k\right)\bar g_0 +
  \sum_{k\in\mathbb N} f_k \psi_k^*\left(\alpha_k^2 g_\Sigma\right). \]
  The surface $\bigl(\Mf,g_0\bigr)$ has infinite area since the truncated cigars we are pasting in each have area bounded below by some uniform positive constant (in fact their area is converging to infinity):
  \begin{align*}
    \Vol_{g_0} \Mf &\ge \sum_{k\in\mathbb N} \Vol_{g_0} V_k
    = \sum_{k\in\mathbb N} \alpha_k^2
    \Vol_{g_\Sigma} \bigl(\gBall_{g_\Sigma}(0;\alpha^{-1}_kR_k)\big)\\
    &\ge 2\pi \sum_{k\in\mathbb N} \alpha_k\bigl(R_k - \alpha_k\bigr)
    \ge 2\pi \sum_{k\in\mathbb N} \frac{R_1-\frac1k}k = \infty,
    &\text{using \eqref{eq:cigar-area-cyl-r}}
  \end{align*}
  and so Theorem \ref{thm:ic-existence} provides an immortal,
  instantaneously complete Ricci flow $\bigl(g(t)\bigr)_{t\in[0,\infty)}$ with
  $g(0)=g_0$. To see that the curvature of $g(t)$ is unbounded at an arbitrary time $t\in [0,\infty)$, note that for any $k\in\N$ with
  $k>t$, we have $t\in [0,T_k]$, so by applying Corollary
  \ref{cor:lower-bd-curv-a} to the flow, considering the $k^{th}$ cigar in our construction, we find that
  \[ \sup_\Mf K[g(t)] \ge \varepsilon\alpha_{k}^{-2} =\varepsilon k^2, \]
  and by letting $k\to\infty$ we find that
  $\sup_\Mf K[g(t)]=\infty$ as claimed.
\end{proof}

\begin{appendix}

  \section{Supporting results}
  The following variant of Perelman's Pseudolocality theorem is due to
  B.-L. Chen. It will be essential to rule out bad behaviour
  of the type in Theorem \ref{thm:sucked_out1}.
  \begin{thm}{\rm\cite[Proposition 3.9]{Che09}}
    \label{thm:chen-pseudo}
    Let $\bigl(g(t)\bigr)_{t\in[0,T]}$ be a smooth Ricci flow
    on a surface $\Mf^2$. If we have for some $p\in\Mf$, $r_0>0$ and $v_0>0$
    \begin{compactenum}[(i)]
    \item $\gBall_{g(t)}(p;r_0)\Subset\Mf$ for all $t\in[0,T]$;
    \item $\Bigl| K[g(0)]\Bigr| \le r_0^{-2}$ on $\gBall_{g(0)}(p;r_0)$;
    \item $\Vol_{g(0)}\gBall_{g(0)}(p;r_0) \ge v_0 r_0^2$,
    \end{compactenum}
    then there exists a constant $C=C(v_0)>0$ such that for all
    $t\in\bigl[0,\min\bigl\{T,\frac1C r_0^2\bigr\}\bigr]$
    \[ \Bigl| K[g(t)] \Bigr| \le 2r_0^{-2}
    \qquad\text{on } \gBall_{g(t)}\Bigl(p;\frac{r_0}2\Bigr). \]
  \end{thm}

  The following isoperimetric inequality due to G. Bol allows us
  estimate the maximum of the curvature on a surface's domain from
  below if we know its area and the length of its boundary. For an
  alternative proof using curvature flows,
  and further generalisations see
  \cite{Top98} and \cite{Top99}.

  \begin{thm}{\rm\cite[eqn.~(30) on p.~230]{Bol41}}
    \label{thm:bol-isop-ineq}
    Let $\Omega$ be a simply-connected domain on a surface
    $\bigl(\Mf^2,g\bigr)$, then
    \begin{equation}
      \label{eq:bol-isop-ineq}
      \bigl(\mathrm{L}_g\,\partial\Omega\bigr)^2 \ge 4\pi \Vol_g(\Omega)
      - \bigl(\Vol_g\Omega\bigr)^2 \sup_\Omega K[g].
    \end{equation}
  \end{thm}

  The following result indicates that a Ricci flow can have extreme properties when it is not assumed to be complete.
  \begin{thm}{\rm(Taken from \cite{Top05})}
    \label{thm:sucked_out2}
    Given $\ep>0$ (however small) there exists a Ricci flow $g(t)$ on the closed unit disc $\overline\Disc\subset\C$ for $t\in [0,\ep)$ (smooth on $[0,\ep)\times \overline\Disc$) such that $g(0)$ is the standard flat metric on the unit disc, but for which
    $$\Vol_{g(t)}(\Disc)\to 0$$
    and
    $$\inf_\Mf K[g(t)]\to\infty$$
    as $t\upto \ep$. Moreover, if we write $g(t)=\ee^{2u}|\dz|^2$ then
    $sup_\Disc (u) \to -\infty$ as $t\upto \ep$.
  \end{thm}

  \begin{proof}
    Let $\m$ be $\Sph^2$, and equip it with an initial metric $g_0$ so
    that $(\m,g_0)$ arises by taking a cylinder $\rSph^1_r\times [-1,1]$,
    where $\rSph^1_r$ is the circle of length $2\pi r$, and capping the
    ends with round hemispheres of the appropriate size:
    \begin{center}
      \includegraphics[width=100mm]{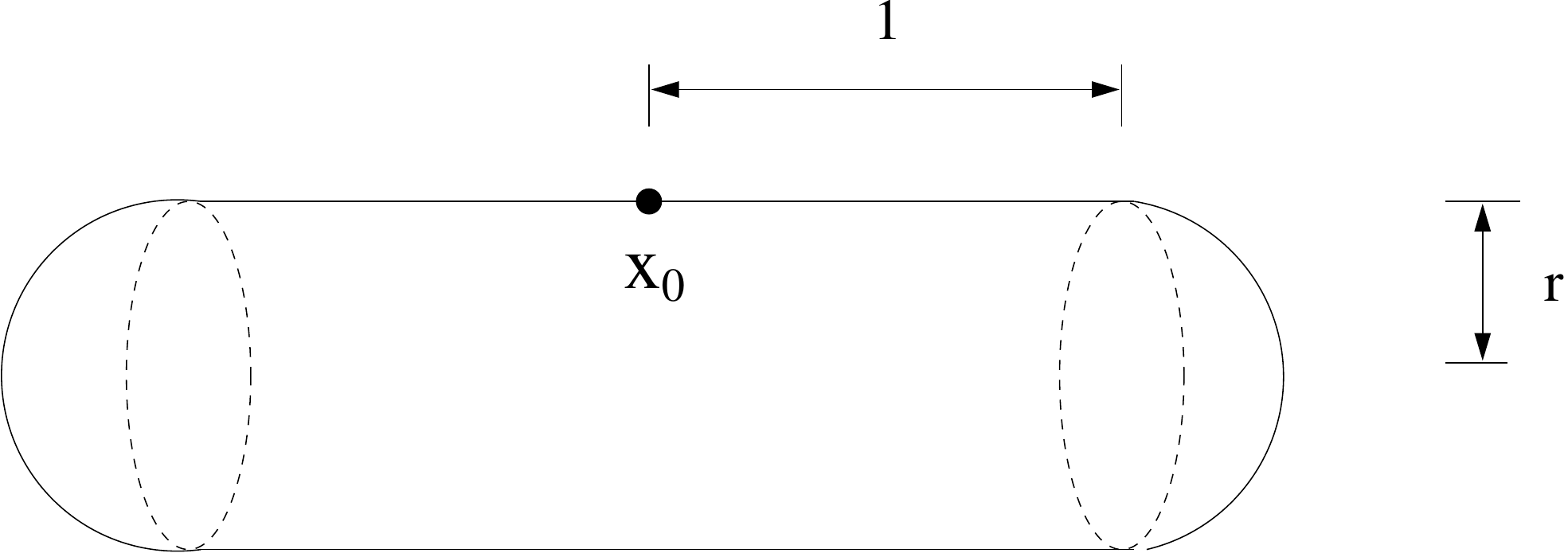}
    \end{center}
    If we take the Ricci flow $g(t)$ starting at $g(0)=g_0$ given by
    Theorem \ref{thm:ic-existence}
    then we know it exists precisely until time
    $\frac{1}{8\pi}\Vol_{g_0}\Sph^2=\half (r + r^2)=\ep$ (for appropriate $r>0$)
    and as $t\upto \ep$, we have
    $$\inf_{\Mf} K[g(t)]\to \infty.$$
    Now pick $x_0\in\m$ mid-way along the
    cylindrical part of $\Sph^2$. Then $g_0$ is flat on $\overline{\gBall_{g_0}(x_0;1)}$.
    Consider the map $F:\overline\Disc\to\Sph^2$ defined by restricting
    the exponential map
    $\exp_{x_0}$, with respect to the metric $g(0)$, where $\overline\Disc$ is
    seen as the closed unit $2$-disc in the tangent space $T_{x_0}\Sph^2$.
    This way, the image of $F$ does not intersect the hemispherical caps
    in the construction of the initial manifold, and is an immersion.

    We construct a Ricci flow $\hat g(t)$ on $\overline\Disc$ by defining
    $\hat g(t)=F^*\bigl(g(t)\bigr)$ at each time. Then $\bigl(\overline\Disc,\hat g(0)\bigr)$ is the flat unit 2-disc, the curvature blows up everywhere at time $\ep=\half (r + r^2)$, and the area decreases to zero at this time.
  \end{proof}

\end{appendix}

\bibliography{2d_ricci_flow}
\end{document}